\documentclass[reqno,twoside,11pt]{amsart}

\usepackage{amsmath,amsfonts,calrsfs,fullpage,amssymb,color}

\newtheorem{Theorem}{Theorem}[section]

\newtheorem{Proposition}[Theorem]{Proposition}
\newtheorem{Lemma}[Theorem]{Lemma}

\newtheorem{Hypothesis}[Theorem]{Hypothesis}
\makeatletter
\@addtoreset{equation}{section}

\makeatother

\def\Q{\mathbb Q}
\def\R{\mathbb R}
\def\N{\mathbb N}

\def\E{\mathbb E}
\def\P{\mathbb P}

\def\eps{\varepsilon}

\newcommand{\esssup}{\operatorname{ess\,sup}}
\newcommand{\essinf}{\operatorname{ess\,inf}}

\newcommand{\one}{1\!\!\!\;\mathrm{l}}

\title[Law of the minimum]{\bf On  the law of the minimum of the solutions to a class of unidimensional SDEs}\date{}

\author[G. Da Prato]{Giuseppe Da Prato}
\address{Scuola Normale Superiore\\
Piazza dei Cavalieri, 7\\ 
56126 Pisa, Italy}
\email{giuseppe.daprato@sns.it}

\author[A. Lunardi]{Alessandra Lunardi}
\address{
Dipartimento di Scienze Matematiche, Fisiche e Informatiche\\
Universit\`a di Parma\\
Parco Area delle Scienze, 53/A\\
43124 Parma, Italy}
\email{alessandra.lunardi@unipr.it}

\author[L. Tubaro]{Luciano Tubaro}
\address{
c/o Dipartimento di Matematica\\
Universit\`a di Trento\\
Via Sommarive 14\\
38123 Povo, Italy}
\email{tubaro@science.unitn.it}



\begin{document}

 \begin{abstract}  
We prove that the law of the minimum $m:=\min_{t\in[0,1]} \xi(t)$ 
of the solution  $\xi$ to a one-dimensional ODE with good nonlinearity has continuous density with respect to the Lebesgue measure. As a byproduct of the procedure, we show that the sets  $ \{ x\in C([0,1]):\; \min x > r\}$ have finite perimeter with respect to the law $\nu$ of the solution $\xi(\cdot)$  in $L^2(0,1)$. 
 \end{abstract}

 \maketitle
 
\section{Introduction}
Let $B(\cdot)$ be a standard Brownian motion in a probability space $(\Omega,\mathcal F,\P)$. We consider a one dimensional SDE, 
\begin{equation}
\label{e1.4}
d\xi(t)=b(\xi)dt+dB(t),\quad \xi(0)=0,
\end{equation}
where $b\in C^2_b(\R)$, the space of the bounded twice differentiable functions with bounded first and second derivative. It is well known that the trajectories $t\mapsto \xi(t)(\omega)$ are continuos for $\P$-a.e. $\omega$.  
The main aim of this note is to prove the following theorem. 

\begin{Theorem}
\label{t1.1}
The law of
$$
m:=\min_{t\in[0,1]} \xi(t)
$$
is absolutely continuous with respect to the Lebesgue measure in $(-\infty, 0)$, with a continuous density.
 \end{Theorem}

The result is well known in the case $b\equiv 0$, where $\xi(t) = B(t)$ for every $t$ and the law of $m$ is given by (e.g.,  \cite[Thm. 6.9]{Schilling})
\begin{equation}
\label{e1.3}
(\P\circ m^{-1})(dr)= \sqrt{\frac{2}{\pi }}\;e^{-\frac{r^2}{2}}\one_{(-\infty,0]}(r) dr. 
\end{equation}

In the case of general $b$, we proceed in two steps. As a first step, 
using the Girsanov theorem, we show that the law  $\nu$ of $\xi(\cdot)$ in $C([0,1])$  is absolutely continuous with respect to the Wiener measure $\P^W$ with a smooth (unbounded) density $\Psi$. 

In the second step we use  the construction of surface integrals of  \cite{BoDaTu} for Gaussian measures in Hilbert spaces. Here the Hilbert space is $L^2(0,1)$, still endowed with the Wiener measure, and we play with the fact that the Wiener measure on $C([0,1])$ is just the restriction of the Wiener measure on the Borel sets of $L^2(0,1)$ to the Borel sets of $C([0,1])$. Such construction yields that  the density of $\Psi \P^W \circ g^{-1}$, with $g(x) = \min x$, is continuous, provided that $\Psi $ belongs to some Sobolev space 
$W^{1,p}(C([0,1]), \P^W)$ for some $p>1$. Our density $\Psi$ is shown to belong to all spaces $W^{1,p}(C([0,1]), \P^W)$ with $p\in[1, +\infty)$, and this allows to conclude. 

In the last section we use the previous results  to show that for every $r<0$ the set $ \{ x\in C([0,1]):\; \min x > r\}$ has finite perimeter with respect to the measure $\nu$. We remark that for $r\geq 0$ the question is not relevant, since such a set is $\nu$-negligible. This gives an example of a nontrivial finite perimeter set with respect to a non-Gaussian (in general,  not log-concave) differentiable measure in an infinite dimensional space.

\section{The main result}

\subsection{Notation and generalities}\label{generalities}

For the general theory of Gaussian measures we refer to \cite{Boga}. Here we recall just the notation used in this paper.

Let  $X$ be  a separable Banach space,  with norm $\|\cdot\|$ (and scalar product $\langle \cdot, \cdot \rangle$, if it is a Hilbert space). The $\sigma$-algebra of the Borel sets in $X$ is denoted by ${\mathcal B}(X)$. 

We consider a centered  Gaussian measure $\mu$ in ${\mathcal B}(X)$, and we denote by $H$  (or by $H_X$, to avoid confusion when different Banach spaces are considered) the corresponding Cameron-Martin space. It is a Hilbert space continuously embedded in $X$, whose scalar product  is denoted by  $\langle h, k\rangle_{H}$. 
If $X$ is a Hilbert space  and $Q$ is the covariance of $\mu$, the Cameron-Martin space coincides with $Q^{1/2}(X)$, and its norm is given by  $\|h\|_{H} = \|Q^{-1/2}h\|$.  

$C^1_b(X)$ denotes the space of all Fr\'echet differentiable $\varphi:X\mapsto \R$, with continuous Fr\'echet derivative $\varphi '$.  For every $h\in H$ there exists a unique element  $\hat{h} \in \overline{X^*}$ (the closure of $X^*$ in $L^2(X, \mu)$) such that 
$$\int_X \frac{\partial \varphi}{\partial h}(x) \mu(dx) = \int_X \varphi(x)\hat{h}(x)\mu(dx), \quad \varphi \in C^1_b(X). $$
Every $\hat{h}$ is a Gaussian random variable,  so that it belongs to $L^p(X, \mu)$ for every $p\in [1, +\infty)$. The Cameron-Martin space is isometric to 
$ \overline{X^*}$, since  $\langle h, k\rangle_{H} := \langle \hat{h}, \hat{k}\rangle_{L^2(X, \mu)}$ for every $h$, $k\in H$. 

If $ \varphi \in C^1_b(X)$, for  every $x\in X$  the mapping $h\mapsto \varphi'(x)(h)$ belongs to $H^*$, and therefore there exists a unique $y\in H$ such that 
$\varphi'(x)(h) = \langle h, y\rangle_{H}$. Such $y$ is denoted by $\nabla_H\varphi(x)$. For every $p\in [1, +\infty)$ the operator $\nabla_H : C^1_b(X)\mapsto L^p(X, \nu;X)$, considered  as an unbounded operator in the space $L^p(X, \mu)$, is closable in $L^p(X, \mu)$. Its closure is still denoted by $\nabla_H$, and the domain of the closure is denoted by $W^{1,p}(X, \mu)$.

If $X$ is a Hilbert space, the symbol $\nabla$ denotes the usual gradient, namely if $ \varphi$ is Fr\'echet differentiable at  $x\in X$, $\nabla \varphi(x)$ is the unique $z\in X$ such that $\varphi'(x)(h) = \langle h,z\rangle $ for every $h\in X$. 
For every $p\in [1, +\infty)$ the operator $Q^{1/2}\nabla  : C^1_b(X)\mapsto L^p(X, \nu;X)$, considered as an unbounded operator in the space $L^p(X, \mu)$, is closable in $L^p(X, \mu)$. Denoting by  $M_p$ its closure, it turns out that $D(M_p) = W^{1,p}(X, \mu)$, and $\nabla_H u = Q^{1/2}M_p u$, for every $u\in 
W^{1,p}(X, \mu)$.  
 
 \vspace{3mm}

In this paper we shall consider the spaces 
$$X:=L^2(0,1), \quad E:= C([0,1])$$
 endowed with the Wiener measure. Usually, the Wiener measure $\P^W$ is considered  in ${\mathcal B}(E)$; equivalent  constructions of it are e.g. in \cite[p. 54-56]{Boga}. In particular, for every Brownian motion $\{W(t): \;0\leq t\leq 1\}$ in any probability space $(\Omega, \mathcal F, \P)$, the image measure $\P \circ W(\cdot)^{-1}$ in ${\mathcal B}(E)$ coincides with $\P^W$. $\P^W$ is centered, Gaussian, and it concentrated in $E_0:=\{ f\in E:\; f(0)=0\}$; the restriction of $\P^W$ to ${\mathcal B}(E _0)$ is a centered nondegenerate Gaussian measure in ${\mathcal B}(E _0)$. 
 
However, in the following we shall use some results about Gaussian measures in Hilbert spaces, and therefore it is convenient to extend $\P^W$ to ${\mathcal B}(X)$. 
Denoting by $i$  the natural immersion  $i:E\mapsto X$, 
the image  measure   $\P^W\circ i^{-1}$ in ${\mathcal B}(X)$ turns out to be   the Gaussian measure $N_Q$ with mean $0$ and covariance operator $Q$ given by 
$$(Qx)(t)=\int_0^1\min\{t,s\}\,x(s)\,ds,\quad x\in  L^2(0,1),\;t\in (0, 1). $$
Since $Q$ is one to one, $N_Q$ is nondegenerate. Moreover,   $i(B)$ is a Borel set in $X$   and $\P^W(B) = N_Q(i(B))$,  for every $B\in {\mathcal B}(E)$; in particular, $N_Q(i(E))=1$. 

As usual, we shall neglect  the immersion $i$ and we shall write $E\subset X$, $N_Q(B) = \P^W(B)$ for  any $B\in {\mathcal B}(E)$. In this sense, $N_Q$ is an extension of $\P^W$ to ${\mathcal B}(X)$.

The corresponding Cameron-Martin spaces $H_E$, $H_X= Q^{1/2}(X)$ do coincide; they are equal to 
$\{ x\in H^1(0,1):\; x(0)=0\}$ (e.g., \cite[Lemmas 2.3.14, 3.2.2]{Boga}). By the characterization of Sobolev spaces through weak derivatives along Cameron-Martin directions, 
(e.g. \cite[Lemma 5.4.7, Theorem 5.7.2]{Boga}), it follows that the Sobolev spaces $W^{1,p}( E, \P^W)$ and $W^{1,p}( X,N_Q)$ coincide for every $p\in [1, +\infty)$.


\subsection{Continuity of the law of $m$}

Set
\begin{equation}
\label{e2.1}
q(t)=\exp \bigg(-\frac12\int_0^tb(\xi(s))^2ds-\int_0^t b(\xi(s))\,dB(s)\bigg),\quad t\geq 0, 
 \end{equation}
 and
 define
$$
\Q(d\omega)=q(1)(\omega)\,\P(d\omega).
 $$
  Since $\E_\P(q(t))=1$ for every $t$, by the Girsanov Theorem (e.g. \cite[Thm. 17.8]{Schilling}), 
 $\{\xi(t): \; t\in [0,1]\}$  is a Brownian motion in $(\Omega,\mathcal F,\Q)$ (notice that  $\Q$
 is a probability measure in $(\Omega,\mathcal F)$). 
  
 Therefore
$\Q\circ (\xi(\cdot)^{-1})=\P^W$
and so, for any $\varphi :E \to\R$ bounded and Borel measurable we have
\begin{equation}
\label{e2.2}
\int_\Omega \varphi (\xi(\cdot)(\omega))\,\Q(d\omega)=\E_\Q(\varphi (\xi(\cdot)))=\int_{E}\varphi (x)\,P^W(dx). 
 \end{equation}
 Since
$$\E_\P(\varphi (\xi(\cdot)))=\E_\Q(\varphi (\xi(\cdot))/q(1) ),$$
we have
\begin{equation}
\label{e2.3}
\int_\Omega \varphi (\xi(\cdot))\,d\P= \int_\Omega \varphi (\xi(\cdot))\,\exp(1/q(1))\,d\Q . 
 \end{equation}

 We define the probability measure $\nu_E$  in ${\mathcal B} (E)$ by 
 \begin{equation}
 \label{eq:nu}
 \nu_E:= \P \circ \xi(\cdot)^{-1}. 
 \end{equation}

\begin{Proposition}
\label{prop:2.1}
We have
\begin{equation}
\label{e2.4}
q(1)=\exp \bigg(-v(\xi(1))+\frac12\int_0^1b^2(\xi(s))ds +\frac12\,\int_0^1b'(\xi(s))ds \bigg),
 \end{equation}
where
\begin{equation}
\label{v}
 v(\eta)=\int_0^\eta b(r)dr,\quad \eta\in \R.
\end{equation}
Therefore, the density of $\nu_E$ with respect to $\P^W$ is the   function
\begin{equation}
\label{psi}
\Psi(x) := \exp\bigg( v(x(1)) -\frac12\int_0^1b^2(x(s))ds -\frac12\,\int_0^1b'(x(s))ds\bigg), \quad x\in E, 
\end{equation}
which belongs to $C^1(E) \cap W^{1,p}(E, \P^W)$ for every $p\in [1, +\infty)$. 
\end{Proposition}
\begin{proof}
By  It\^o's formula we have 
 $$
  d(v\circ \xi)= b(\xi(t))d\xi(t) +\frac12\,b'(\xi(t) )dt = b^2(\xi(t))dt+b(\xi(t) )dB(t)+\frac12\,b'(\xi(t) )dt,
 $$     
so that
$$v(\xi(t))= \int_0^tb^2(\xi(s))ds  +\int_0^tb(\xi(s))dB(s)+\frac12\,\int_0^tb'(\xi(s))ds$$
and \eqref{e2.4} holds. Replacing in  \eqref{e2.3}, for every bounded and Borel measurable $\varphi:E\mapsto \R$   we obtain 
$$
 \int_\Omega \varphi (\xi(\cdot))\,d\P= \int_\Omega \varphi (\xi(\cdot))\,e^{v(\xi(1)) -\frac12\int_0^1b^2(\xi(s))ds -\frac12\,\int_0^1b'(\xi(s))ds}\,d\Q.  $$
Since   $\{\xi(t): \; t\in [0,1]\}$  is a Brownian motion in $(\Omega,\mathcal F,\Q)$, the law $\Q \circ \xi(\cdot)^{-1}$ of $\xi(\cdot)$ in ${\mathcal B}(E)$ 
is equal to $\P^W$. Therefore, 
%
 \begin{equation}
 \label{e2.4b}
 \int_\Omega \varphi (\xi(\cdot))\,d\P= \int_{E} \varphi (x)\,e^{v(x(1))-\frac12\int_0^1b^2(x(s))ds -\frac12\,\int_0^1b'(x(s))ds}\,\P^W(dx), 
 \end{equation}
namely, $\nu$ is absolutely continuous with respect to $\P^W$, with density $\Psi$ given by  \eqref{psi}. Notice that $|\Psi(x)| \leq C_1 \exp (C_2\|x\|_{\infty})$, for some $C_1$, $C_2>0$, and therefore $\Psi\in L^p(E, \P^W)$ for every $p\in [1, +\infty)$, although it is not bounded. 
 $\Psi$ is obviously $C^1$, with Fr\'echet derivative given by 
$$\Psi'(x)(y):= \Psi(x) \left( b(x(1))y(1) - \int_0^1 b(x(s))b'(x(s))y(s)ds - \frac12\,\int_0^1b''(x(s))y(s)ds  \right) , \quad x, \;y\in E. $$
Moreover, there is $C_3>0$ such that 
$$\| \Psi'(x)\|_{C([0,1])^*} \leq C_3 \Psi(x), \quad x\in E. $$
By \cite[Lemma 5.7.10]{Boga}, $\Psi  \in W^{1,p}( E, \P^W)$ for every $p\in [1, +\infty)$. 
\end{proof}

By the considerations at the end of subsection \ref{generalities}, $\Psi$ also belongs to $ W^{1,p}(X, N_Q)$ for every $p\in [1, +\infty)$. 

We want to use now some results of \cite{BoDaTu} that were proved under the following hypothesis, called ``local Malliavin condition".

\begin{Hypothesis}
\label{h2.1}
Let $X$ be a separable Hilbert space endowed with a nondegenerate centered Gaussian measure $\mu$, let $g\in W^{1,p}(X, \mu)$ for every $p\in [1, +\infty)$, and let $I\subset \R$ be an open interval.  Assume  that there are two random variables $U:X\mapsto X$, $\gamma:X\mapsto \R$ such that $ \langle Mg(x),U(x)  \rangle =\gamma(x)$, for a.e.  $x\in g^{-1}(I)$, $\gamma(x) \neq 0$ for a.e.  $x\in g^{-1}(I)$, and
$U/\gamma \in D(M_p^*)\quad \forall\;p\in [1, +\infty)$. 
\end{Hypothesis}

Lemma 2.2 and Proposition 3.3 of \cite{BoDaTu} yield the following propositions. 

\begin{Proposition}
\label{prop:BoDaTu1}
Let  Hypothesis \ref{h2.1} be satisfied. If $\Psi\in W^{1,q} (X, \mu)$ for some $q>1$, the function
\begin{equation}
\label{F} 
F_{\Psi}(r):= \int_{\{x:\; g(x) \geq r\} }\Psi(x) \mu(dx), 
\end{equation}
is continuously differentiable at any $r\in I$. 
\end{Proposition}

\begin{Proposition}
\label{prop:BoDaTu2} Let $X= L^2(0,1)$ be endowed with the Wiener measure $ N_Q$. Then  the function $h(x):= \esssup_{0\leq s\leq 1} x(s)$ satisfies Hypothesis \ref{h2.1} in every interval $I\subset (0, +\infty)$. 
\end{Proposition}

Propositions \ref{prop:2.1}, \ref{prop:BoDaTu1}, \ref{prop:BoDaTu2} are the tools to prove Theorem \ref{t1.1}. 

\vspace{5mm}
\noindent
{\em Proof of Theorem  \ref{t1.1}.}
Let $X= L^2(0,1)$ be endowed with the Wiener measure $N_Q$, and set
$$g(x) := \essinf x, \quad x\in X. $$
Since $g(x) = - h(-x)$ with  $h(x) = \esssup_{0\leq s\leq 1} x(s)$, $g$ satisfies Hypothesis \ref{h2.1}  in every interval $I\subset (-\infty, 0)$, by Proposition  \ref{prop:BoDaTu2}. Applying Proposition \ref{prop:BoDaTu1} to the constant function $\Psi \equiv 1$, we obtain that the level sets $\{x\in X:\, g(x) =r\}$ are $N_Q$-negligible for every $r\in \R$. Applying it to the
function $\Psi$ defined in \eqref{psi}, we obtain that  
$$F_{\Psi}(r) = \int_{\{x: \, g(x) \geq r\}} \Psi(x) N_Q(dx)$$
is continuously differentiable in $(-\infty, 0)$. Therefore,   for $a<b<0$ we have $F_{\Psi}(a) - F_{\Psi}(b)= \int_{\{x: \, g(x) \in [a,b)\}} \Psi(x) N_Q(dx)
= \int_{\{x: \, g(x) \in [a,b]\}} \Psi(x) N_Q(dx)$. 
  
On the other hand, by Proposition \ref{prop:2.1}  for $a<b<0$ we have
$$\int_{\{\omega\in \Omega:\, m(\omega) \in [a,b]\}} d\P = \int_{\{x\in E:\, g(x) \in [a, b]\}}\Psi(x) \P^W(dx)$$
and therefore 
$$\int_{\{\omega\in \Omega:\, m(\omega) \in [a,b]\}}  d\P = \int_{\{x\in X:\, g(x)\in [a, b]\}}\Psi(x) N_Q(dx) = F_{\Psi}(a) - F_{\Psi}(b) = - \int_a^b F_{\Psi}'(r)dr . $$
This implies that  the law of $m$ is absolutely continuous with respect to the Lebesgue measure in $(-\infty, 0)$, with continuous density $-F_{\Psi}' $.

\section{An application to geometric measure theory}

Fixed any $r\in \R$, we consider here the sets 
$$O_r:= \{ x\in L^2(0,1): \; \essinf x >r\}, \quad O_{r } \cap E := \{ x\in C([0,1]): \;  \min x >r\}. $$
They have different topological properties, since $O_{r } \cap E$ is open in $E$ while $O_r$ is neither open or closed  and it has empty interior  in $X$. 
However they differ by a $N_Q$-negligible set, so that  they have the same measure theoretic properties as far as the measures $N_Q$ and $\nu :=\Psi N_Q$ are considered. We are going to make this sentence more precise, dealing with perimeters.

\subsection{Notation and basic results}
The notion of perimeter and of perimeter measure is well known for Gaussian centered nondegenerate measures in separable Banach spaces (\cite{F,FH,AMMP})
and it has been extended more recently to more general classes of differentiable measures, see e.g.  \cite{AmDaGoPa,RZZ1,BR,BPR,RZZ2,DPL}. Here we use the definitions and some results  of \cite{DPL}, restricted to Gaussian and weighted Gaussian measures that is the case under consideration here.

Let $\mu$ be a centered  Gaussian measure in a separable Hilbert space $X$, and let $w$ be a nonnegative weight function, such that 
\begin{equation}
\label{w}
w, \; \log w \in \bigcap_{p>1}W^{1,p}(X, \mu). 
\end{equation}
The weighted measure $\nu$ is defined as 
\begin{equation}
\label{nu}
\nu(dx) = w(x)\mu(dx). 
\end{equation}
The operator $Q^{1/2}\nabla : C^1_b(X)\subset L^p(X, \nu)\mapsto L^p(X, \nu;X)$ is closable, as an unbounded operator in $L^p(X, \nu)$, for every $p\in [1, +\infty)$. Its closure is denoted by $M_{p,\nu}$, and the domain of the closure is denoted by $W^{1,p}(X, \nu)$. For every $z\in X$ and $\varphi\in C^1_b(X)$, the integration formula 
\begin{equation}
\label{parti_introd}
 \int_X \langle M_{p, \nu}u, z\rangle \varphi  \,d\nu = -  \int_X u \langle Q^{1/2}\nabla \varphi,z \rangle + \int_Xu\, v_z\varphi\,d\nu 
 \end{equation}
holds, with $v_z(x) = (Q^{1/2}z)^{\land}(x) + M_{p,\mu}(\log w)$.  By \eqref{w}, $v_z\in L^p(X, \nu)$ for every $p\in [1, +\infty)$, although it is not bounded. (For more details and further properties, see \cite{TAMS,Simone}).  

By \eqref{parti_introd}, good vector fields with finite dimensional range, $F(x)= \sum_{i=1}^n f_i(x) z_i$ with $f_i\in C^1_b(X)$ and $z_i\in X$, belong to the domain of the adjoint operator $M^*_{p,\nu}$ for every $p>1$, and $M^*_{p,\nu}F(x) = - \sum_{i=1}^n (\langle Q^{1/2}\nabla f_i,z_i \rangle - v_{z_i}f_i)$ is independent of $p$. We denote by $\widetilde{C}^{1 }_b(X, X)$ the set of all such vector fields, and by $M^*F$ the common value of $M^*_{p,\nu}F$ for all $p>1$. 
The total variation of any $u\in L^q(X, \nu)$ with $q>1$ is defined by 
\begin{equation}
\label{V(u)} 
V(u)  :=   \sup \bigg\{ \int_X u \,M^* F \, d\nu: \; F\in \widetilde{C}^1_b(X ,X), \; \|F(x)\| \leq 1\; \forall x\in X\bigg\}
\end{equation}
The condition  $V(u)<+\infty$ is equivalent to the existence of  a Borel $X$-valued vector measure $m$ such that, setting $m_z(B) := \langle m(B), z\rangle$ for every $z\in X$ and for every Borel set $B\subset X$, we have 
\begin{equation}
\label{defBVintro}
 \int_X u (\langle Q^{1/2}\nabla \varphi,z \rangle - v_z\varphi)\,d\nu= - \int_X  \varphi  \,dm_z, \quad z\in X, \;\varphi\in C^1_b(X).
 \end{equation}
If $u$ satisfies one of such equivalent conditions, we say that it has bounded variation, and we write $u\in BV(X,\nu)$. 
If  for some Borel set $B$ the function $u= \one_B$ has bounded variation, the total variation measure $|m|$ is called perimeter measure and $|m|(X)$ is called perimeter of $B$. 

Notice that if two Borel sets $B_1$, $B_2$ differ by a negligible set (namely, $\nu (B_1\setminus B_2) = \nu(B_2\setminus B_1) =0$), $\one_{B_1}$ is of bounded variation if and only if $\one_{B_2}$ is of bounded variation, and in this case the perimeter measures of $B_1$ and of $B_2$ do coincide. 
If $\nu(B) =0$, or $\nu(B)=1$, we have $V(\one_B)=0$ and the vector measure  $m $ is trivial, $m\equiv 0$. So, the notion of perimeter is meaningful only if $\nu(B)\in (0,1)$. 

Of course we can take $w\equiv 1$ in \eqref{w}, \eqref{nu}. With this choice, the notion of bounded variation function with respect to the Gaussian measure $\mu$ coincides with the one of \cite{F,AMMP}, although different notations are used. We recall that if $X$ is a separable Banach space endowed with a centered non-degenerate Gaussian measure $\mu$, the definition of total variation of \cite{F,AMMP} is 
\begin{equation}
\label{tildeV(u)} 
\widetilde{V}_{X,\mu}(u)  :=   \sup \bigg\{ \int_X u \, \text{div}_\mu \widetilde{F} \, d\nu: \;  \widetilde{F}\in \widetilde{C}^1_b(X ,H), \; \|F(x)\| _H\leq 1\; \forall x\in X\bigg\}
\end{equation}
where the space $\widetilde{C}^1_b(X ,H)$ consists of the vector fields $\widetilde{F}$ of the type $\widetilde{F}= \sum_{i=1}^n f_i(x) h_i$ with $f_i\in C^1_b(X)$ and $h_i\in H$, for some $n\in \N$.  The Gaussian divergence div$_{\mu}$ is still defined by duality: given a vector field $\Phi\in L^1(X, \mu; H)$, a function $\beta\in L^1(X, \mu)$ is called Gaussian divergence of $\Phi$, and denoted by  div$_{\mu}\Phi$, if 
$$\int_X \langle \nabla_H \varphi, \Phi\rangle_Hd\mu = - \int_X \varphi\,\beta \,d\mu, \quad \varphi\in C^1_b(X). $$
In the case that $X$ is a Hilbert space, we have  $\widetilde{F}\in \widetilde{C}^1_b(X ,H)$ iff $F= Q^{-1/2}\widetilde{F}\in \widetilde{C}^1_b(X ,X)$, and the Gaussian divergence div$_{\mu}\widetilde{F}$ is equal to $-M^*_{p, \mu}F$ for every $p>1$. Therefore, $V(u) = \widetilde{V}_{X,\nu}(u) $ for every $u\in L^q(X, \mu)$ with $q>1$. 
See \cite{DPL}.

\subsection{Perimeters of the sets $O_r$}

As in Section 2, we set here $X= L^2(0,1)$, $E=C([0,1])$, $H=\{ f\in H^1(0,1):\; h(0) =0\}$. $X$ and $E$ are endowed with the Wiener measure, and $H$ is their common Cameron-Martin space.   

By \eqref{e1.3}, $\P^W(\{ x\in E: \, \min(x) \geq 0\} )= N_Q(\{ x\in X: \, \essinf (x) \geq 0\} ) =0$, and therefore $\nu(\{ x\in X: \, \essinf (x) \geq 0\} ) =0$. Therefore we consider the sets $O_r$, $O_r\cap E$ only for $r<0$. 

\vspace{3mm}

As far as the perimeter of $O_r\cap E$ with respect to $\P^W $ is concerned, we recall that the general theory of BV functions in Banach spaces endowed with Gaussian measures has been developed only for nondegenerate Gaussian measures. We already remarked that $\P^W $ is degenerate in $E$, while its restriction to $E_0= \{ f\in E: \; f(0)=0\}$ is non degenerate. This is why in the next considerations we replace $E$ by $E_0$. 

$O_r\cap E_0$ is a convex open set in $E_0$. Since $\P^W$ is a nondegenerate centered Gaussian measure in $E_0$, by \cite[Prop.4.2]{CLMN} we have $\P^W (\partial O_r) =\P^W\{x\in E_0: \,\min x =r\} =0$, and  the perimeter (with respect to $\P^W$, in the sense of \cite{F,FH,AMMP})  of $O_r\cap E_0$ is finite.  The same argument cannot be applied to $O_r$ (considered as a subset of $X$ with the Wiener measure) because $O_r$ has empty interior part in $X$.  
However, since $E_0$ is continuously embedded in $X$, every $F\in C^1_b(X, H)$ belongs also to $C^1_b(E_0, H)$ (to be more precise: the restriction to $E_0$ of any $F\in C^1_b(X, H)$ belongs  to $C^1_b(E_0, H)$), and since $E_0$ is dense in $X$, $\sup_{x\in X}\|F(x)\|_H = \sup_{x\in E_0}\|F(x)\|_H $. Therefore, 
$\widetilde{V}_{X,N_Q}(u) \leq \widetilde{V}_{E_0, \P^W}(u) $ for every $u\in L^q(X, N_Q) = L^q(E_0, \P^W)$. In particular, taking $u = \one_{O_r}$, we obtain that $O_r$ has finite perimeter with respect to $\nu$. 
An expression of the related integration formula \eqref{defBVintro} along suitable $z$ may be found in \cite{BZ}. 

\vspace{5mm}

Let us consider now  the measure $\nu$. We shall prove that the sets $O_r$ have finite perimeter with respect to $\nu$; to this aim we use the next lemma. 

\begin{Lemma}
The function $g(x):= \essinf x$ belongs to $W^{1,p}(X, \nu)$ for every $p\in [1, +\infty)$. The set of all $x\in E$ that have a unique minimum point has full measure, and for every $x\in E$ having a unique minimum point $\tau_x$ we have
\begin{equation}
\label{eq:min}
M_{p, \nu}g (x) =  \one_{[0, \tau_x]}. 
\end{equation}
\end{Lemma}
\begin{proof}  
Setting $h(x)= \esssup x$, we already know that $h\in W^{1,p}(X, N_Q)$ for every $p\geq 1$, and so does $g$ since $g(x)= -h(-x)$. 

It is well known that $N_Q(\{ x\in E: \; x$ has a unique maximum point$\} ) =1$; moreover, 
   for every $x\in E$ having a unique maximum point $\eta_x$ we have $M_{p,N_Q}h(x) =  \one_{[0, \eta_x]}$, (e.g.,  \cite[Prop. 3.2]{BoDaTu}). It follows  that $g\in W^{1,p}(X, N_Q)$ for every $p\geq 1$, and for every $x\in E$ having a unique maximum point $\tau_x$ we have $M_{p,N_Q}g(x) =  \one_{[0, \tau_x]}$. 
 
We recall that $\nu = \Psi N_Q$, where the weight $\Psi$ belongs to $W^{1,p}(X, N_Q)$ for every $p\in [1, +\infty)$, and also its logarithm
$$\log \Psi (x) = v(x(1)) -\frac12\int_0^1b^2(x(s))ds -\frac12\,\int_0^1b'(x(s))ds$$
does. By \cite[Prop. 3.5]{Simone}, $g\in W^{1,p}(X, \nu)$ for every $p$, and $M_{p,\nu} g(x)  =M_{p,N_Q}g(x)$, for $\nu$-a.e. $x\in X$. The statement follows. 
\end{proof}

%
%
Now we are ready to prove the main result of this section.

\begin{Proposition}
For every $r<0$, $O_r$ has finite $\nu$-perimeter. 
\end{Proposition}
\begin{proof}
Let us estimate $V(\one_{O_r})$. We claim that for every $F\in  \widetilde{C}^1_b(X ,X)$ we have
\begin{equation}
\label{eq:limite}
\int_{O_r} M^* F\,d\nu = \lim_{\eps \to 0^+} \frac{1}{\eps} \int_{ g^{-1}([r,r+\eps])} \langle M  g, F\rangle\,d\nu, 
\end{equation}
where we set $Mg =  M_{p,\nu}g$ for every $p\geq 1$. 
Once \eqref{eq:limite} is established,  we estimate its right hand side recalling that by \eqref{eq:min} we have $\|M g(x)\|\leq 1$ for $\nu$-a.e.$x$, and therefore
$$\left|  \frac{1}{\eps}\int_{ g^{-1}([r,r+\eps])} \langle M  g, F\rangle\,d\nu \right| \leq  \frac{1}{\eps}\, \nu( g^{-1}([r,r+\eps])) \|F\|_{L^{\infty}(X, X)}.$$
On the other hand, 
$$\frac{1}{\eps} \,\nu( g^{-1}([r,r+\eps]) )= \frac{1}{\eps} \int_{g^{-1}([r,r+\eps]) }\Psi(x) N_Q(dx)$$
has finite limit $l$ as $\eps \to 0$, by Proposition  \ref{prop:BoDaTu1} (using the notation of Prop.  \ref{prop:BoDaTu1}, we have $l=-F_{\Psi}'(r)$). Therefore, 
$$\int_{O_r} M^* F\,d\nu \leq l \|F\|_{L^{\infty}(X, X)}, $$
which implies that $V(\one_{O_r})  <+\infty$.

To prove that  \eqref{eq:limite} holds we use an argument taken from \cite[Prop. 3.8]{Lincei}, that we reproduce here adapting it to the present context. For $\eps >0$ we use the approximation of 
 $\one_{[r, +\infty)}$ given by 
$$
\theta_\eps (\xi)=\left\{\begin{array}{ll}
0, & \xi\le r\\
(\xi-r)/\eps,
& \xi\in [r ,r+\eps ]\\
1, &  \xi\ge r+\eps 
\end{array}\right.
$$
Then $\theta_\epsilon \circ g\in W^{1,p}(X, \nu)$ for every $p\in [1, +\infty)$, and
\begin{equation}
\label{eq:comp}
M_{p,\nu}(\theta_\epsilon \circ g) = (\theta_\epsilon' \circ g) \,M_{p,\nu}g = \frac{1}{\eps} \one_{g^{-1}([r, r+\eps])}M_{p,\nu}g, 
\end{equation}
see e.g. \cite[Lemma 2.2]{TAMS} (notice that the sets $g^{-1}(r)$ and $g^{-1}(r+\eps)$ are $N_Q$-negligible and therefore $\nu$-negligible).

Now, for every $F\in D(M^*_{p,\nu})$ and for every $\varphi \in W^{1,p}(X, \nu)$ we have 
$$\int_X \langle M_{p,\nu}\varphi, F\rangle d\nu = \int_X \varphi\, M^*_{p,\nu}F\,d\nu . $$
Taking $F\in  \widetilde{C}^1_b(X ,X)$,  $\varphi = \theta_\epsilon \circ g$, and recalling \eqref{eq:comp} we get 
$$ \frac{1}{\epsilon}\int_{g^{-1}([r, r+\eps])}\langle M g,F  \rangle\, d\nu=\int_{X} (\theta_\epsilon \circ g)\,M^*(F)\,d\nu.  $$
As    $\epsilon\to 0$ the right hand side converges to $\int_{g^{-1}([r, +\infty))}  M^*(F)\,d\nu = \int_{O_r} M^*(F)\,d\nu$, and so does the left hand side. Therefore,  \eqref{eq:limite} holds. 
\end{proof}

\section{Acknowledgements}
The first author would like to thank the Isaac Newton Institute for Mathematical Sciences for support and hospitality 
during the programme ``Scaling limits, rough paths, quantum field theory'' when part of the work on this paper was undertaken.
This work was supported by EPSRC Grant Number EP/R014604/1 and by the PRIN research project  2015233N54 ``Deterministic and stochastic evolution equations" .

\end{document}